\documentclass[a4paper,11pt]{article}

\usepackage{amsmath}
\usepackage{amssymb}
\usepackage{amsthm}
\usepackage{graphicx}
\usepackage{amscd}
\usepackage{epic, eepic}
\usepackage{url}
\usepackage{color}
\usepackage[utf8]{inputenc} 
\usepackage{comment}
\usepackage{enumerate}   
\usepackage{todonotes}
\usepackage{graphicx}
\usepackage{epstopdf}
\usepackage{enumitem}
\usepackage{tikz}
\usepackage[top=22mm, bottom=23mm, left=22mm, right=22mm]{geometry}
\usepackage[bookmarks=false,hidelinks]{hyperref}

\makeatletter
\renewenvironment{proof}[1][\proofname] {\par\pushQED{\qed}\normalfont\topsep6\p@\@plus6\p@\relax\trivlist\item[\hskip\labelsep\bfseries#1\@addpunct{.}]\ignorespaces}{\popQED\endtrivlist\@endpefalse}
\makeatother

\newtheorem{theorem}{\bf Theorem}[section]
\newtheorem{lemma}[theorem]{\bf Lemma}
\newtheorem{corollary}[theorem]{\bf Corollary}

\newtheorem{conjecture}[theorem]{\bf Conjecture}

\newtheorem{problem}[theorem]{\bf Problem}

\theoremstyle{definition}

\newtheorem{remark}[theorem]{\bf Remark}
\newtheorem{definition}[theorem]{\bf Definition}

\def\eps{\varepsilon}

\title{On locally rainbow colourings}

\author{Barnab\'as Janzer\thanks{Department of Pure Mathematics and Mathematical Statistics, University of Cambridge, United Kingdom. Research supported by EPSRC DTG. Email: \textbf{janzer.barnabas@gmail.com}.}
	\and
Oliver Janzer\thanks{Department of Pure Mathematics and Mathematical Statistics, University of Cambridge, United Kingdom. Research supported by a fellowship at Trinity College. Email: \textbf{oj224@cam.ac.uk}.}}
\date{}

\begin{document}

\maketitle

\begin{abstract}
    Given a graph $H$, let $g(n,H)$ denote the smallest $k$ for which the following holds. We can assign a $k$-colouring $f_v$ of the edge set of $K_n$ to each vertex $v$ in $K_n$ with the property that for any copy $T$ of $H$ in $K_n$, there is some $u\in V(T)$ such that every edge in $T$ has a different colour in $f_u$.
    
    The study of this function was initiated by Alon and Ben-Eliezer. They characterized the family of graphs $H$ for which $g(n,H)$ is bounded and asked whether it is true that for every other graph $g(n,H)$ is polynomial. We show that this is not the case and characterize the family of connected graphs $H$ for which $g(n,H)$ grows polynomially. Answering another question of theirs, we also prove that for every $\eps>0$, there is some $r=r(\eps)$ such that $g(n,K_r)\geq n^{1-\eps}$ for all sufficiently large $n$.
    
    Finally, we show that the above problem is connected to the Erdős--Gyárfás function in Ramsey Theory, and prove a family of special cases of a conjecture of Conlon, Fox, Lee and Sudakov by showing that for each fixed $r$ the complete $r$-uniform hypergraph $K_n^{(r)}$ can be edge-coloured using a subpolynomial number of colours in such a way that at least $r$ colours appear among any $r+1$ vertices. \end{abstract}

\section{Introduction}

\subsection{Local rainbow colourings} Estimating the minimum possible size of a program that computes specific Boolean functions is a major research area in Theoretical Computer Science. In 1993, Karchmer \cite{Kar93} introduced the so-called fusion method for finding circuit lower bounds. This technique unifies and generalizes the topological method of Sipser \cite{Sip84} and the approximation method of Razborov \cite{Raz89}. Karchmer and Wigderson \cite{KW93,Wig93} demonstrated that proving lower bounds for circuit sizes can be reduced to extremal combinatorics problems, and Wigderson \cite{Wig93} presented three problems that arise this way. One of them is as follows.

\begin{problem}[Karchmer and Wigderson \cite{Wig93}] \label{prob:Wig}
    Given a positive integer $n$, estimate the smallest $k$ for which the following is true. There exist colourings $c_1,\dots,c_n$ of the $n$-dimensional cube $\{0,1\}^n$ with $k$ colours such that for any three distinct $x,y,z\in \{0,1\}^n$ there is a coordinate $i\in [n]$ such that $x_i$, $y_i$ and $z_i$ are not all equal and the three colours $c_i(x)$, $c_i(y)$ and $c_i(z)$ are pairwise distinct.
\end{problem}

Karchmer and Wigderson \cite{KW93} proved that $k$ has to grow with $n$; more precisely that $k$ needs to be at least $\Omega(\frac{\log \log^* n}{\log \log \log^* n})$, where $\log^*n$ is the smallest integer $m$ such that applying the function $\log_2(x)$ iteratively $m$ times, starting with input $n$, one obtains a number not exceeding $1$.

Alon and Ben-Eliezer \cite{AB11} improved this significantly by showing that $k$ needs to be at least $\Omega((\frac{\log n}{\log \log n})^{1/4})$. As part of their approach, they introduced the following problem, which will be our main focus in this paper.

\begin{definition}
    Let $n$ be a positive integer and let $H$ be a graph. For each vertex $v$ of a given clique $K_n$, let $f_v$ be a (not necessarily proper) colouring of the edges of the same $K_n$. We say that the collection of these $n$ colourings is $(n,H)$-local if for any copy $T$ of $H$ in $K_n$, there exists some $u\in V(T)$ such that all edges of $T$ receive different colours in $f_u$.
\end{definition}

\begin{problem}[Alon and Ben-Eliezer \cite{AB11}] \label{prob:AB}
    Let $g(n,H)$ be the smallest $k$ for which there is a collection of colourings, each using $k$ colours, which is $(n,H)$-local. Estimate the growth of $g(n,H)$ as $n\rightarrow \infty$. 
\end{problem}

To see the connection to Problem \ref{prob:Wig}, note that $g(n,P_3)$ is a lower bound for the smallest possible $k$ in the Karchmer--Wigderson problem. (Here and below, $P_{\ell}$ denotes the path with $\ell$ edges.) Indeed, we can think of the $n$ coordinates of $\{0,1\}^n$ as the $n$ vertices of $K_n$ and the elements of Hamming weight 2 in $\{0,1\}^n$ as edges in $K_n$. A valid collection of colourings in Problem \ref{prob:Wig} is then necessarily an $(n,P_3)$-local colouring, since we may choose $x,y,z$ to be three sets in $\{0,1\}^n$ which correspond to the three edges of some $P_3$. Alon and Ben-Eliezer showed that $g(n,P_3)=\Omega((\frac{\log n}{\log \log n})^{1/4})$, implying the same lower bound for the problem of Karchmer and Wigderson.

Alon and Ben-Eliezer also studied Problem \ref{prob:AB} for general graphs $H$. They characterized the family of graphs for which $g(n,H)$ is bounded.

\begin{theorem}[Alon and Ben-Eliezer \cite{AB11}] \label{thm:bounded}
    For a fixed graph $H$, there is a constant $c(H)$ such that $g(n,H)\leq c(H)$ for every $n$ if and only if $H$ contains at most $3$ edges and $H$ is neither $P_3$ nor $P_3$ together with any number of isolated vertices. Moreover, in all these cases $g(n,H)\leq 5$ for every $n$.
\end{theorem}

They used the local lemma to obtain the following general upper bound.

\begin{theorem}[Alon and Ben-Eliezer \cite{AB11}] \label{thm:local lemma bound}
    Let $H$ be a fixed graph with $r$ vertices. Then $g(n,H)=O(r^4n^{1-\frac{2}{r}})$.
\end{theorem}

They also proved polynomial lower bounds for various small graphs and used this to obtain the following result.

\begin{theorem}[Alon and Ben-Eliezer \cite{AB11}] \label{thm:13 edges}
    For any graph $H$ with at least 13 edges, there is a constant $b=b(H)>0$ such that $g(n,H)=\Omega(n^b)$.
\end{theorem}

They posed three concrete open problems in their paper.

\begin{enumerate}
    \item Improve the bounds for $g(n,P_3)$.

    \item Is it true that if $g(n,H)$ is unbounded, then it grows polynomially?

    \item Is it true that for every $\eps>0$ there is some $r=r(\eps)$ such that $g(n,K_r)\geq n^{1-\eps}$ for every sufficiently large $n$?
\end{enumerate}

The first problem is well motivated by its connection to Problem \ref{prob:Wig}. The second one is motivated by Theorem \ref{thm:13 edges}. The third one is motivated by Theorem \ref{thm:local lemma bound} and the fact that if $H'$ is a subgraph of $H$ on the same set of vertices, then $g(n,H')\leq g(n,H)$.

Some progress on these questions was made by Cheng and Xu \cite{CX22}. They showed that $g(n,P_{\ell})$ is polynomial in $n$ for every $\ell\geq 4$. Combined with other new bounds for small graphs, they used this to prove that if $H$ is a graph with at least $6$ edges, then $g(n,H)$ is polynomial, improving Theorem~\ref{thm:13 edges}. Finally, they showed that $g(n,K_r)=\Omega(n^{2/3})$ holds for all $r\geq 8$, which can be seen as progress towards answering the third question above.

In this paper, we answer the second and third question of Alon and Ben-Eliezer, and also show that $g(n,P_3)$ grows subpolynomially, which essentially answers the first question as well.

\begin{theorem} \label{thm:P3}
    We have $g(n,P_3)=n^{o(1)}$.
\end{theorem}

Together with the lower bound $g(n,P_3)=\Omega((\frac{\log n}{\log \log n})^{1/4})$ of Alon and Ben-Eliezer, Theorem \ref{thm:P3} answers the second question of theirs in the negative. Our next result answers their third question affirmatively.

\begin{theorem} \label{thm:cliques}
    For each $\ell\geq 2$, we have $g(n,C_{2\ell})=\Omega(n^{1-\frac{2}{\ell+1}})$. Consequently, for any even $r\geq 4$, $g(n,K_r)=\Omega\left(n^{1-\frac{4}{r+2}}\right)$.
\end{theorem}

\begin{remark} \label{remark:odd cliques}
    Using a variant of the proof of Theorem \ref{thm:cliques}, we can also prove that for each sufficiently large odd $r$, we have $g(n,K_r)=\Omega\left(n^{1-\frac{10}{r-3}}\right)$, showing that the exponent tends to $1$ in this case as well.
\end{remark}

We also obtain a near-complete characterization of the family of graphs $H$ for which $g(n,H)$ is polynomial. The only graph $H$ for which we cannot decide whether $g(n,H)$ is polynomial is the disjoint union of a $P_3$ and a $P_1$ (together with an arbitrary number of isolated vertices).

\begin{theorem} \label{thm:characterize}
    Let $H$ be a graph which is not the disjoint union of $P_3$ and $P_1$ together with an arbitrary number of isolated vertices. Then there exists some $b=b(H)>0$ such that $g(n,H)=\Omega(n^b)$ if and only if $H$ has at least $5$ edges or $H$ has precisely $4$ edges and is triangle-free.
\end{theorem}

In particular, we obtain a full characterization of the family of connected graphs $H$ for which $g(n,H)$ is polynomial.

\begin{corollary} \label{cor:characterize connected}
    Let $H$ be a connected graph. Then there exists some $b=b(H)>0$ such that $g(n,H)=\Omega(n^b)$ if and only if $H$ has at least $4$ edges and $H$ is different from the triangle with a pendant edge.
\end{corollary}

\subsection{The Erdős--Gyárfás function}\label{subsec:IntroEGy}

We will see (in Section~\ref{sec:upper}) that local rainbow colourings for certain graphs $H$ are related to the Erdős--Gyárfás function in Ramsey Theory (especially to Theorem~\ref{theorem_EGy43} below). In this subsection we describe the Erdős--Gyárfás problem, and state a new result resolving a family of special cases of a conjecture of Conlon, Fox, Lee and Sudakov~\cite{conlon2015grid,conlon2015recent}.

\begin{definition}
    Let $p,q,r,n\geq 2$ be positive integers with $q\leq \binom{p}{r}$. An edge-colouring of the $r$-uniform complete hypergraph $K_n^{(r)}$ is a $(p,q)$-colouring if at least $q$ distinct colours appear among any $p$ vertices. Let $f_r(n,p,q)$ be the smallest positive integer $k$ such that there exists a $k$-colouring of the edges of $K_n^{(r)}$ forming a $(p,q)$-colouring.
\end{definition}

The function $f_r(n,p,q)$ was introduced by Erdős and Shelah~\cite{erdos1975problems}, and first studied in  more detail by Erdős and Gyárfás~\cite{erdHos1997variant} (for $r=2$). Let us first consider the graph case $r=2$. When $q=2$, then a $(p,q)$-colouring is simply a colouring which avoids monochromatic sets of size $p$, so as a special case we get the classical multicolour Ramsey problem. In particular, $f_2(n,3,2)$ (and hence $f_2(n,p,2)$) is at most logarithmic in $n$. On the other extreme, when $q=\binom{p}{2}$, we trivially have $f_2(n,p,\binom{p}{2})=\binom{n}{2}$ (as long as $p\geq 4$). The function $f_r(n,p,q)$ is clearly increasing in $q$, and Erdős and Gyárfás~\cite{erdHos1997variant} investigated how the behaviour of $f_2(n,p,q)$ changes as $q$ increases from $2$ to $\binom{p}{2}$. Among other results, they proved that when $p=q$, it is polynomial in $n$, i.e., $f_2(n,p,p)=\Omega(n^{\alpha_p})$ for some $\alpha_p>0$. They asked if this is the smallest value of $q$ for which $f_2(n,p,q)$ is polynomial in $n$, i.e., whether or not $f_2(n,p,p-1)$ is subpolynomial in $n$.

The first difficult case $p=4$ was settled by Mubayi~\cite{mubayi1998edge}, who gave a construction showing that $f_2(n,4,3)=n^{o(1)}$. This was first extended to $p=5$ as well by Eichhorn and Mubayi~\cite{eichhorn2000}, and later Conlon, Fox, Lee and Sudakov~\cite{conlon2015erdHos} proved that $f_2(n,p,p-1)$ is subpolynomial for all $p$, fully answering this question of Erdős and Gyárfás.
\begin{theorem}[Conlon, Fox, Lee and Sudakov~\cite{conlon2015erdHos}]\label{theorem_EGygraphs}
	For any fixed $p\geq 4$, we have $$f_2(n,p,p-1)\leq e^{(\log{n})^{1-1/(p-2)+o(1)}}=n^{o(1)}.$$
\end{theorem}

Consider now the Erdős--Gyárfás function for general uniformity $r$ (this is the setting in which Erdős and Shelah~\cite{erdos1975problems} originally introduced the problem). Answering a question of Graham, Rothschild and Spencer~\cite{graham1990ramsey}, Conlon, Fox, Lee and Sudakov~\cite{conlon2015grid} proved that $f_3(n,4,3)$ is subpolynomial in $n$. (This has close connections to the proof of Shelah~\cite{shelah1988primitive} of primitive recursive bounds for the Hales--Jewett theorem, see~\cite{conlon2015grid} for details.)

\begin{theorem}[Conlon, Fox, Lee and Sudakov~\cite{conlon2015grid}]\label{theorem_EGy43}
We have $$f_3(n,4,3)\leq e^{(\log n)^{2/5+o(1)}}=n^{o(1)}.$$
\end{theorem}

Conlon, Fox, Lee and Sudakov~\cite{conlon2015grid} also proved that $f_r(n,p,\binom{p-1}{r-1}+1)$ is at least polynomial in $n$. In light of this result, as well as Theorem~\ref{theorem_EGygraphs} and Theorem~\ref{theorem_EGy43}, they proposed the following conjecture.

\begin{conjecture}[Conlon, Fox, Lee and Sudakov~\cite{conlon2015grid,conlon2015recent}]\label{conjecture_EGy}
	For any positive integers $p$ and $r$ with $2\leq r<p$, $$f_r\left(n,p,\binom{p-1}{r-1}\right)=n^{o(1)}.$$
\end{conjecture}

Note that the case $r=2$ holds by Theorem~\ref{theorem_EGygraphs}, and Theorem~\ref{theorem_EGy43} is the case $r=3$, $p=4$. As further evidence towards Conjecture~\ref{conjecture_EGy}, Conlon, Fox, Lee and Sudakov~\cite{conlon2015grid} proved that its statement holds for $r=3$, $p=5$ as well.
These results (i.e., $r=2$; or $r=3$ and $p\in\{4,5\}$) are the only known cases of Conjecture~\ref{conjecture_EGy}. We show that Conjecture~\ref{conjecture_EGy} holds whenever $p=r+1$.

\begin{theorem}\label{theorem_EGygeneral}
	For any $r\geq 3$, we have
	$$f_r(n,r+1,r)\leq e^{(\log n)^{2/5+o(1)}}=n^{o(1)}.$$
\end{theorem}

In addition to proving a family of special cases of Conjecture~\ref{conjecture_EGy}, Theorem~\ref{theorem_EGygeneral} is also directly related to the problem of Karchmer and Wigderson (Problem~\ref{prob:Wig}): in Subsection~\ref{subsection_connectiontoWig} we briefly describe how Theorem~\ref{theorem_EGygeneral} implies that a natural approach to finding polynomial lower bounds for Problem~\ref{prob:Wig} cannot work.\medskip

The rest of the paper is organized as follows. In Section \ref{sec:upper}, we prove Theorem \ref{thm:P3} and the ``only if" part of Theorem \ref{thm:characterize}, that is, a subpolynomial upper bound for $g(n,H)$ whenever $H$ has at most $3$ edges or has precisely $4$ edges and contains a triangle. In Section \ref{sec:lower}, we prove Theorem \ref{thm:cliques} and the ``if" part of Theorem \ref{thm:characterize}. In Section~\ref{sec:EGy} we prove Theorem~\ref{theorem_EGygeneral}. We finish the paper by giving some brief concluding remarks in Section \ref{sec:concluding}.

\section{Upper bounds} \label{sec:upper}
In this section we give constructions providing subpolynomial upper bounds for $g(n,H)$ when $H$ is one of the following graphs:
\begin{itemize}
    \item $P_3$, the path with $3$ edges;
    \item $T_p$, a triangle with a pendant edge; or
    \item $T_e$, the disjoint union of a triangle and an edge.
\end{itemize}


Note that if $H'$ is formed from $H$ by adding some isolated vertices, then any collection of colourings which is $(n,H)$-local is also $(n,H')$-local, hence $g(n,H')\leq g(n,H)$. So subpolynomiality for the $3$ graphs above, together with Theorem~\ref{thm:bounded}, deals with all cases of the ``only if'' part of Theorem~\ref{thm:characterize}.
Furthermore, since $P_3$ is a subgraph of $T_p$ (on the same vertex set), the result that $g(n,T_p)$ is subpolynomial easily implies that $g(n,P_3)$ is subpolynomial. Nevertheless, we will first focus on the proof for $P_3$, as it is slightly simpler and motivates the construction for $T_p$ (and also interesting on its own due to its connection to Problem~\ref{prob:Wig}).\medskip

We will pick our colourings in such a way that in each colouring $f_v$, the edges which contain $v$ receive different colours from the ones that do not contain $v$. Note that if this property holds, then the only way a $P_3$ $abcd$ can be non-rainbow (i.e., have some colour appearing more than once) for each colouring corresponding to its vertices if $$f_a(bc)=f_a(cd), f_b(ab)=f_b(bc), f_c(bc)=f_c(cd)\textnormal{ and }f_d(ab)=f_d(bc).$$

It is not difficult to prove that for any collection of colourings with $n^{o(1)}$ colours we can find a $P_3$ in which the first three of the four equalities above hold, i.e., the colourings $f_a,f_b$ and $f_c$ of the $P_3$ are all non-rainbow. However, we will show that we can construct a collection of colourings with $n^{o(1)}$ colours such that any $P_3$ $abcd$ is rainbow in either $f_a$ or $f_d$.

The main idea is the following. Assume that the colourings $f_v$ are defined in such a way that any edge $\{x,y\}$ not containing $v$ is coloured by the colour $\gamma(\{v,x,y\})$, where $\gamma$ is some colouring of the edges of the complete $3$-uniform hypergraph formed by our $n$ vertices. Then the condition $f_a(bc)=f_a(cd)$ becomes $\gamma(abc)=\gamma(acd)$, and the condition $f_d(ab)=f_d(bc)$ becomes $\gamma(abd)=\gamma(bcd)$. Thus, if both of these conditions hold, then $\gamma$ uses at most $2$ colours on the $4$ vertices $a,b,c,d$.
However, recall from Subsection~\ref{subsec:IntroEGy} the following result of Conlon, Fox, Lee and Sudakov~\cite{conlon2015grid} about the Erdős--Gyárfás problem.


\begin{theorem}[Conlon, Fox, Lee and Sudakov~\cite{conlon2015grid}]\label{thm:hypergraphEGy}
    The edges of the complete $3$-uniform hypergraph $K_n^{(3)}$ on $n$ vertices can be coloured using $e^{(\log n)^{2/5+o(1)}}=n^{o(1)}$ colours in such a way that at least $3$ different colours appear among the four edges spanned by any $4$ vertices.
\end{theorem}


We are ready to prove that $g(n,P_3)$ is subpolynomial, which follows easily from the discussion above.

\begin{proof}[Proof of Theorem~\ref{thm:P3}]
    Using Theorem~\ref{thm:hypergraphEGy}, we can take a colouring $\gamma$ of the complete $3$-uniform hypergraph formed on our $n$ vertices such that $\gamma$ uses at most $e^{(\log n)^{2/5+o(1)}}$ colours and at least $3$ colours appear among any $4$ vertices. We now define our collection of colourings as follows. Let $z_0$ be a colour not used by $\gamma$ (i.e., $z_0\not \in \operatorname{Im}(\gamma)$). Define, for any vertex $v$ and edge $e=\{x,y\}$,
    \begin{equation*}
        f_v(e)=
        \begin{cases}
        z_0 & \textnormal{if $v\in e$}\\
        \gamma(e\cup\{v\}) & \textnormal{if $v\not \in e$}.
        \end{cases}
    \end{equation*}

Let $abcd$ be any $P_3$ in our $K_n$. As noted in the discussion above, if the edges of this $P_3$ do not all receive different colours in $f_a$, then $\gamma(abc)=\gamma(acd)$. Similarly, if the edges of the $P_3$ do not all receive different colours in $f_d$, then $\gamma(abd)=\gamma(bcd)$. But we know that at least $3$ different colours appear among $\gamma(abc),\gamma(abd),\gamma(acd)$ and $\gamma(bcd)$, so $\gamma(abc)=\gamma(acd)$ and $\gamma(abd)=\gamma(bcd)$ cannot simultaneously hold. Hence our collection of colourings is $(n,P_3)$-local with $1+e^{(\log n)^{2/5+o(1)}}=n^{o(1)}$ colours.
\end{proof}

We now turn to the proof that $g(n,T_p)$ is subpolynomial, where $T_p$ is the triangle with a pendant edge. As noted before, this result is stronger than Theorem~\ref{thm:P3}, and correspondingly the proof will be a refinement of the one above.
\begin{theorem}\label{thm:pendant}
    We have $g(n,T_p)=n^{o(1)}$.
\end{theorem}
\begin{proof}
    As before, take an edge-colouring $\gamma$ of the $3$-uniform complete hypergraph $K_n^{(3)}$ using $n^{o(1)}$ colours such that among any $4$ vertices at least $3$ colours appear. Furthermore, take an edge-colouring $\delta$ of the clique $K_n$ using $O(\log n)$ colours such that there is no monochromatic triangle. (It is well-known that this is possible -- for example, label the vertices by elements of $\{0,1\}^m$ and colour the edge $xy$ by the minimal $i$ for which $x_i\not =y_i$.) We may assume that $\gamma$ and $\delta$ have disjoint images.
    
    We define our collection of colourings $f_v$ by setting

        \begin{equation*}
        f_v(e)=
        \begin{cases}
        \delta(e) & \textnormal{if $v\in e$}\\
        \gamma(e\cup\{v\}) & \textnormal{if $v\not \in e$}.
        \end{cases}
    \end{equation*}

Take any copy $abcd$ of $T_p$: the vertices $b,c,d$ form a triangle and $a$ is joined to $b$. We show that this copy must be rainbow in one of $f_a$, $f_c$ or $f_d$.

If this copy of $T_p$ is not rainbow under $f_c$, then we must have either $\delta(bc)=\delta(cd)$ or $\gamma(abc)=\gamma(bcd)$. Similarly, if the copy is not rainbow under $f_d$, then either $\delta(bd)=\delta(cd)$ or $\gamma(abd)=\gamma(bcd)$. Note that, by the definition of $\delta$, we cannot have both $\delta(bc)=\delta(cd)$ and $\delta(bd)=\delta(cd)$. Thus, without loss of generality, we have $\gamma(abc)=\gamma(bcd)$.

But if our copy of $T_p$ is not rainbow under $f_a$, then at least two of $abc,abd,acd$ have the same colour in $\gamma$. Together with $\gamma(abc)=\gamma(bcd)$, this would imply that at most two colours appear among $a,b,c,d$ in $\gamma$, giving a contradiction. So our collection of colourings is $(n,T_p)$-local with $n^{o(1)}+O(\log{n})=n^{o(1)}$ colours.
\end{proof}

Finally, we prove that $g(n,T_e)$ is subpolynomial, where $T_e$ denotes the disjoint union of a triangle and an edge. In fact, we will prove a logarithmic bound here.

\begin{theorem}\label{thm:triangle+edge}
    We have $g(n,T_e)=O(\log n)$.
\end{theorem}
\begin{proof}
    We may assume that the vertices of our clique $K_n$ are elements of $\{0,1\}^m$, where $m=\lceil\log_2 n\rceil$. Given vertices $x$ and $y$, let $\delta(xy)=\min\{i:x_i\not=y_i\}$ be the first coordinate where $x$ and $y$ differ. Observe that $\delta$ has the property that for any three vertices $x,y,z$, exactly two of $\delta(xy),\delta(xz),\delta(yz)$ are equal. Moreover, if $\delta(xz)=\delta(yz)$ then $\delta(xy)>\delta(xz)=\delta(yz)$.

    Define our collection of colourings $f_v$ as follows. For any edge $xy$, let         \begin{equation*}
        f_v(xy)=
        \begin{cases}
        -\delta(xy) & \textnormal{if $v\in \{x,y\}$}\\
        \max\{\delta(vx),\delta(vy)\}& \textnormal{if $v\not \in \{x,y\}$}.
        \end{cases}
    \end{equation*}
    (The only purpose of the minus sign in the first case is to ensure that $f_v$ takes different values on edges that contain $v$ and on edges that do not.)
Consider any copy of $T_e$ formed by a triangle $abc$ and an edge $xy$; we show that this copy is rainbow under one of $f_a, f_b$ or $f_c$. Without loss of generality, we may assume that $\delta(ab)>\delta(ac)=\delta(bc)$.

Note that $f_a(ab)\not=f_a(ac)$. So if this copy of $T_e$ is non-rainbow under $f_a$, then we must have $f_a(xy)=f_a(bc)$, i.e., $\max\{\delta(ax),\delta(ay)\}=\max\{\delta(ab),\delta(ac)\}=\delta(ab)$. Without loss of generality, we have $\delta(ab)=\delta(ax)$. But then $\delta(bx)>\delta(ab)$ and hence $\max\{\delta(bx),\delta(by)\}>\delta(ab)$. This implies that $f_b(ac)\not =f_b(xy)$. As $f_b(ab)\not=f_b(bc)$, our copy of $T_e$ must be rainbow under $f_b$, finishing the proof.
\end{proof}

\section{Lower bounds} \label{sec:lower}

\subsection{The proof of Theorem \ref{thm:cliques}}

In this subsection, we prove Theorem \ref{thm:cliques}. The proof uses the following lemma of the second author {\cite[Theorem 3.1]{Jan23}}, which is a significant generalization of the Bondy--Simonovits theorem \cite{BS74}.

\begin{lemma} \label{lem:good cycles}
    Let $\ell \geq 2$ and $s$ be positive integers. Then there exists a constant $C = C(\ell, s)$ with the following property. Suppose that $G = (V, E)$ is a graph with $N$ vertices and at least $CN^{1+1/\ell}$ edges. Let $\sim$ be a symmetric binary relation on $V$ such that for every $u \in V$ and $v \in V$, $v$ has at most $s$ neighbours $w \in V$ which satisfy $u \sim w$. Then $G$ contains a $2\ell$-cycle $x_1x_2 \dots x_{2\ell}$ such that $x_i \not \sim x_j$ for every $i \neq j$.
\end{lemma}

\begin{proof}[Proof of Theorem \ref{thm:cliques}]
    Let $C=C(\ell,1)$ be provided by Lemma \ref{lem:good cycles} and let $c=(4C)^{-\frac{\ell}{\ell+1}}$. For each vertex $v\in V(K_n)$, let $f_v$ be an edge-colouring of $E(K_n)$ which uses $[k]$ as colours, where $k\leq cn^{1-\frac{2}{\ell+1}}$. Define an auxiliary graph $G$ whose vertex set is $V=V(K_n)\times [k]$ and in which $(u,i)$ and $(v,j)$ are joined by an edge if and only if $u\neq v$,  $f_u(uv)=i$ and $f_v(uv)=j$. Observe that there is a natural bijection between the edges of $G$ and the edges of $K_n$, so $e(G)=e(K_n)=\binom{n}{2}$. Moreover, clearly, the number of vertices in $G$ is $N=nk$. Hence, $$CN^{1+1/\ell}\leq C(cn^{2-\frac{2}{\ell+1}})^{1+1/\ell}=Cc^{1+1/\ell}n^2=n^2/4\leq \binom{n}{2}.$$
    For vertices $(u,i),(v,j)\in V(G)$, let us write $(u,i)\sim (v,j)$ if $u=v$. We claim that if $x,y\in V(G)$, then $y$ has at most one neighbour $z$ in $G$ such that $x\sim z$. Indeed, let $u$ be the first coordinate of $x$ and let $v$ be the first coordinate of $y$. Then (as $x\sim z$) the first coordinate of $z$ must be $u$, and (as $z$ is a neighbour of $y$ in $G$) the second coordinate of $z$ has to be the colour of the edge $uv$ in the colouring $f_u$. Hence, by Lemma \ref{lem:good cycles} applied with $s=1$, it follows that there is a $2\ell$-cycle $x_1x_2\dots x_{2\ell}$ in $G$ such that the first coordinate of each $x_i$ is different. Let $x_i=(u_i,\alpha_i)$. Then, for each $i$, we have $f_{u_i}(u_{i-1}u_i)=f_{u_i}(u_i,u_{i+1})=\alpha_i$, where indices are considered mod $2\ell$. Therefore, the $2\ell$-cycle $u_1u_2\dots u_{2\ell}$ in $K_n$ witnesses that the collection of colourings $\{f_v:v\in V(K_n)\}$ is not $(n,C_{2\ell})$-local. Hence, any $(n,C_{2\ell})$-local colouring must use more than $cn^{1-\frac{2}{\ell+1}}$ colours, which means that $g(n,C_{2\ell})> cn^{1-\frac{2}{\ell+1}}$.

    The second assertion of Theorem \ref{thm:cliques} follows trivially since $K_r$ contains $C_r$ as a subgraph on the same vertex set, so $g(n,K_r)\geq g(n,C_r)$.
\end{proof}

When $r$ is odd, we can use a variant of the above method to prove the bound stated in Remark~\ref{remark:odd cliques}. Let $\theta_{\ell,t}$ be the union of $t$ paths of length $\ell$ which share the same endpoints but are pairwise internally vertex-disjoint. Note that $\theta_{\ell,2}=C_{2\ell}$. A result of the second author {\cite[Theorem 3.7]{Jan23}} shows that Lemma \ref{lem:good cycles} can be generalized to find, under the same conditions (with a constant $C$ that depends in addition on $t$) a copy of $\theta_{\ell,t}$ without a pair of vertices related by $\sim$. Using this result, an argument very similar to the proof of Theorem \ref{thm:cliques} shows that if $f_v$ are $k$-colourings of the edge set of $K_n$ for $k\leq cn^{1-\frac{2}{\ell+1}}$ where $c$ is a sufficiently small constant, then we can find a copy $T$ of $\theta_{\ell,3}$ in $K_n$ with the property that for each $v \in V(T)$, the colour of every edge of $T$ incident to $v$ is the same in $f_v$. It follows that $g(n,\theta_{\ell,3}\cup C_{2q})=\Omega(n^{1-\frac{2}{\min(\ell,q)+1}})$, where $\theta_{\ell,3}\cup C_{2q}$ is the disjoint union of a $\theta_{\ell,3}$ and a $C_{2q}$. Choosing $\ell\in \{\lfloor\frac{r+1}{5}\rfloor,\lfloor\frac{r+1}{5}\rfloor-1\}$ such that $\ell$ is even and setting $q=\frac{r+1-3\ell}{2}$, we obtain $q\geq \ell\geq \frac{r-8}{5}$ and $|V(\theta_{\ell,3}\cup C_{2q})|=3\ell-1+2q=r$, so $g(n,K_r)\geq g(n,\theta_{\ell,3}\cup C_{2q})\geq \Omega(n^{1-\frac{2}{\ell+1}})\geq \Omega(n^{1-\frac{10}{r-3}})$.

\subsection{The proof of Theorem \ref{thm:characterize}}

In this subsection, we complete the proof of Theorem \ref{thm:characterize}. Our results from Section \ref{sec:upper} together with Theorem \ref{thm:bounded} already prove the ``only if" part of Theorem \ref{thm:characterize}, so it suffices to prove the ``if" part, which amounts to giving a polynomial lower bound for $g(n,H)$ in the remaining cases.

We remark that if $g(n,H)$ is polynomial, then so is $g(n,H^+)$, where $H^+$ is the graph obtained from $H$ by adding an isolated vertex. Indeed, assume that there are $k$-colourings $f_v$ of the edges of $K_n$ for each $v\in V(K_n)$ which form an $(n,H^+)$-local collection. Let $s=|V(H)|+1$ and let $u_1,\dots,u_s$ be arbitrary distinct vertices in $K_n$. Now for each $v\in V(K_n)$, define the colouring $f'_v$ of $E(K_n)$ by setting $f'_v(e)=(f_v(e),f_{u_1}(e),\dots,f_{u_s}(e))$. This is a colouring which uses at most $k^{s+1}$ colours. We claim that these colourings form an $(n,H)$-local colouring. Indeed, otherwise we could find a copy $T$ of $H$ in $K_n$ such that for each $v\in V(T)$ there are at least two edges in $T$ which have the same colour in $f'_v$. By definition, any such pair of edges have the same colour in all of $f_v,f_{u_1},\dots,f_{u_s}$, so, as $s>|V(H)|$, we may find a copy $T^+$ of $H^+$ (obtained by adding a vertex $u_i$ to $T$) such that for each $v\in V(T^+)$ there are two edges in $T^+$ which have the same colour in $f_v$. This contradicts the assumption that the $f_v$ are $(n,H^+)$-local. Hence, the $f'_v$ are indeed $(n,H)$-local, so $g(n,H)\leq g(n,H^+)^{s+1}$.

Together with our earlier observation, it also follows that if $g(n,H)$ is polynomial and $F$ contains $H$ as a (not necessarily spanning) subgraph, then $g(n,F)$ is also polynomial.

By the above discussion, it suffices to consider graphs with no isolated vertices. We first focus on graphs $H$ with precisely $4$ edges and no isolated vertices; the list of such graphs can be found in Table \ref{table:4edgegraphs}. It was shown in \cite{CX22} that $g(n,C_4)=\Omega(n^{1/3})$ and that $g(n,P_4)=\Omega(n^{1/5})$. When $H$ contains a triangle, then $g(n,H)$ is subpolynomial, and when $H$ is the union of $P_3$ and $P_1$, we do not know whether $g(n,H)$ is polynomial or not. The remaining $6$ cases are all covered by the following definition and theorem.

\begin{table}

\centering

\begin{tabular}{|c|c|c|}

\hline

\begin{minipage}{0.3\linewidth}
\vspace{5mm}
\centering
    \begin{tikzpicture}[scale=0.35]
	
	\draw[fill=black](0,0)circle(5pt);
        \draw[fill=black](3,0)circle(5pt);
        \draw[fill=black](3,3)circle(5pt);
        \draw[fill=black](0,3)circle(5pt);

        \draw[thick](0,0)--(3,0)--(3,3)--(0,3)--(0,0);
	
    \end{tikzpicture}
\vspace{5mm}
\end{minipage}%
&
\begin{minipage}{0.3\linewidth}
\vspace{5mm}
\centering
    \begin{tikzpicture}[scale=0.35]
	
	\draw[fill=black](0,0)circle(5pt);
        \draw[fill=black](3,0)circle(5pt);
        \draw[fill=black](6,0)circle(5pt);
        \draw[fill=black](1.5,2.6)circle(5pt);

        \draw[thick](0,0)--(3,0)--(1.5,2.6)--(0,0)(3,0)--(6,0);	
	
    \end{tikzpicture}
\vspace{5mm}
\end{minipage}%
&
\begin{minipage}{0.3\linewidth}
\vspace{5mm}
  \centering
    \begin{tikzpicture}[scale=0.35]
	
	\draw[fill=black](0,0)circle(5pt);
        \draw[fill=black](3,0)circle(5pt);
        \draw[fill=black](6,0)circle(5pt);
        \draw[fill=black](9,0)circle(5pt);
        \draw[fill=black](1.5,2.6)circle(5pt);

        \draw[thick](0,0)--(3,0)--(1.5,2.6)--(0,0)(6,0)--(9,0);	
	
    \end{tikzpicture}
\vspace{5mm}
\end{minipage}%
\\

\hline

\begin{minipage}{0.3\linewidth}
\vspace{5mm}
  \centering
    \begin{tikzpicture}[scale=0.35]
	
	\draw[fill=black](0,0)circle(5pt);
        \draw[fill=black](4,3)circle(5pt);
        \draw[fill=black](4,1)circle(5pt);
        \draw[fill=black](4,-1)circle(5pt);
        \draw[fill=black](4,-3)circle(5pt);

        \draw[thick](0,0)--(4,3)(0,0)--(4,1)(0,0)--(4,-1)(0,0)--(4,-3);	
    \end{tikzpicture}

\vspace{5mm}
\end{minipage}%
&
\begin{minipage}{0.3\linewidth}
\vspace{5mm}
  \centering
    \begin{tikzpicture}[scale=0.35]
	
	\draw[fill=black](0,0)circle(5pt);
        \draw[fill=black](3,2)circle(5pt);
        \draw[fill=black](3,0)circle(5pt);
        \draw[fill=black](3,-2)circle(5pt);
        \draw[fill=black](6,3)circle(5pt);

        \draw[thick](0,0)--(3,2)--(6,3)(0,0)--(3,0)(0,0)--(3,-2);	
    \end{tikzpicture}
\vspace{5mm}
\end{minipage}%
&
\begin{minipage}{0.3\linewidth}
\vspace{5mm}
  \centering
    \begin{tikzpicture}[scale=0.35]
	
	\draw[fill=black](0,0)circle(5pt);
        \draw[fill=black](3,2)circle(5pt);
        \draw[fill=black](3,0)circle(5pt);
        \draw[fill=black](3,-2)circle(5pt);
        \draw[fill=black](6,2)circle(5pt);
        \draw[fill=black](6,-2)circle(5pt);

        \draw[thick](0,0)--(3,2)(0,0)--(3,0)(0,0)--(3,-2)(6,2)--(6,-2);	
    \end{tikzpicture}
\vspace{5mm}
\end{minipage}%

\\

\hline

\begin{minipage}{0.3\linewidth}
\vspace{5mm}
\centering
    \begin{tikzpicture}[scale=0.35]
	
	\draw[fill=black](0,0)circle(5pt);
        \draw[fill=black](3,0)circle(5pt);
        \draw[fill=black](6,0)circle(5pt);
        \draw[fill=black](9,0)circle(5pt);
        \draw[fill=black](12,0)circle(5pt);

        \draw[thick](0,0)--(3,0)--(6,0)--(9,0)--(12,0);
	
    \end{tikzpicture}
\vspace{5mm}
\end{minipage}%
&
\begin{minipage}{0.3\linewidth}
\vspace{5mm}
\centering
    \begin{tikzpicture}[scale=0.35]
	
	\draw[fill=black](0,0)circle(5pt);
        \draw[fill=black](3,0)circle(5pt);
        \draw[fill=black](6,0)circle(5pt);
        \draw[fill=black](9,0)circle(5pt);
        \draw[fill=black](3,-3)circle(5pt);
        \draw[fill=black](6,-3)circle(5pt);

        \draw[thick](0,0)--(3,0)--(6,0)--(9,0)(3,-3)--(6,-3);
	
    \end{tikzpicture}
\vspace{5mm}
\end{minipage}%
&
\begin{minipage}{0.3\linewidth}
\vspace{5mm}
  \centering
    \begin{tikzpicture}[scale=0.35]
	
	\draw[fill=black](0,0)circle(5pt);
        \draw[fill=black](3,0)circle(5pt);
        \draw[fill=black](6,0)circle(5pt);
        \draw[fill=black](0,3)circle(5pt);
        \draw[fill=black](3,3)circle(5pt);
        \draw[fill=black](6,3)circle(5pt);

        \draw[thick](0,0)--(3,0)--(6,0);	
        \draw[thick](0,3)--(3,3)--(6,3);	
	
    \end{tikzpicture}
\vspace{5mm}
\end{minipage}%
\\

\hline

\begin{minipage}{0.25\linewidth}
\vspace{5mm}
  \centering
    \begin{tikzpicture}[scale=0.35]
	
	\draw[fill=black](0,0)circle(5pt);
        \draw[fill=black](3,0)circle(5pt);
        \draw[fill=black](6,0)circle(5pt);
        \draw[fill=black](1.5,-3)circle(5pt);
        \draw[fill=black](4.5,-3)circle(5pt);
        \draw[fill=black](1.5,-6)circle(5pt);
        \draw[fill=black](4.5,-6)circle(5pt);

        \draw[thick](0,0)--(3,0)--(6,0)(1.5,-3)--(4.5,-3)(1.5,-6)--(4.5,-6);	
    \end{tikzpicture}

\vspace{5mm}
\end{minipage}%
&
\begin{minipage}{0.25\linewidth}
\vspace{5mm}
  \centering
    \begin{tikzpicture}[scale=0.35]
	
	\draw[fill=black](0,0)circle(5pt);
        \draw[fill=black](3,0)circle(5pt);
        \draw[fill=black](6,0)circle(5pt);
        \draw[fill=black](9,0)circle(5pt);
        \draw[fill=black](0,3)circle(5pt);
        \draw[fill=black](3,3)circle(5pt);
        \draw[fill=black](6,3)circle(5pt);
        \draw[fill=black](9,3)circle(5pt);

        \draw[thick](0,0)--(3,0)(6,0)--(9,0)(0,3)--(3,3)(6,3)--(9,3);	
    \end{tikzpicture}
\vspace{5mm}
\end{minipage}%
&
\begin{minipage}{0.25\linewidth}
\vspace{5mm}
  \centering
    \begin{tikzpicture}[scale=0.35]
	
    \end{tikzpicture}
\vspace{5mm}
\end{minipage}%

\\

\hline

\end{tabular}

\caption{The list of graphs with $4$ edges and no isolated vertices}
\label{table:4edgegraphs}

\end{table}

\begin{table}

\centering

\begin{tabular}{|c|c|c|}

\hline

\begin{minipage}{0.3\linewidth}
\vspace{5mm}
  \centering
    \begin{tikzpicture}[scale=0.35]
	
	\draw[fill=black](0,0)circle(5pt);
        \draw[fill=black](4,3)circle(5pt);
        \draw[fill=black](4,1)circle(5pt);
        \draw[fill=black](4,-1)circle(5pt);
        \draw[fill=black](4,-3)circle(5pt);

        \draw[thick,red](0,0)--(4,3)(0,0)--(4,1);
        \draw[thick,blue](0,0)--(4,-1)(0,0)--(4,-3);	
    \end{tikzpicture}

\vspace{5mm}
\end{minipage}%
&
\begin{minipage}{0.3\linewidth}
\vspace{5mm}
  \centering
    \begin{tikzpicture}[scale=0.35]
	
	\draw[fill=black](0,0)circle(5pt);
        \draw[fill=black](3,2)circle(5pt);
        \draw[fill=black](3,0)circle(5pt);
        \draw[fill=black](3,-2)circle(5pt);
        \draw[fill=black](6,3)circle(5pt);

        \draw[thick,red](0,0)--(3,2)--(6,3);
        \draw[thick,blue](0,0)--(3,0)(0,0)--(3,-2);	
    \end{tikzpicture}
\vspace{5mm}
\end{minipage}%
&
\begin{minipage}{0.3\linewidth}
\vspace{5mm}
  \centering
    \begin{tikzpicture}[scale=0.35]
	
	\draw[fill=black](0,0)circle(5pt);
        \draw[fill=black](3,2)circle(5pt);
        \draw[fill=black](3,0)circle(5pt);
        \draw[fill=black](3,-2)circle(5pt);
        \draw[fill=black](6,2)circle(5pt);
        \draw[fill=black](6,-2)circle(5pt);

        \draw[thick,red](0,0)--(3,2)(6,2)--(6,-2);
        \draw[thick,blue](0,0)--(3,0)(0,0)--(3,-2);
    \end{tikzpicture}
\vspace{5mm}
\end{minipage}%

\\

\hline

\begin{minipage}{0.3\linewidth}
\vspace{5mm}
  \centering
    \begin{tikzpicture}[scale=0.35]
	
	\draw[fill=black](0,0)circle(5pt);
        \draw[fill=black](3,0)circle(5pt);
        \draw[fill=black](6,0)circle(5pt);
        \draw[fill=black](0,3)circle(5pt);
        \draw[fill=black](3,3)circle(5pt);
        \draw[fill=black](6,3)circle(5pt);

        \draw[thick,blue](0,0)--(3,0)--(6,0);	
        \draw[thick,red](0,3)--(3,3)--(6,3);	
	
    \end{tikzpicture}
\vspace{5mm}
\end{minipage}%
&
\begin{minipage}{0.25\linewidth}
\vspace{5mm}
  \centering
    \begin{tikzpicture}[scale=0.35]
	
	\draw[fill=black](0,0)circle(5pt);
        \draw[fill=black](3,0)circle(5pt);
        \draw[fill=black](6,0)circle(5pt);
        \draw[fill=black](1.5,-3)circle(5pt);
        \draw[fill=black](4.5,-3)circle(5pt);
        \draw[fill=black](1.5,-6)circle(5pt);
        \draw[fill=black](4.5,-6)circle(5pt);

        \draw[thick,red](0,0)--(3,0)--(6,0);
        \draw[thick,blue](1.5,-3)--(4.5,-3);
        \draw[thick,blue](1.5,-6)--(4.5,-6);	
    \end{tikzpicture}

\vspace{5mm}
\end{minipage}%
&
\begin{minipage}{0.25\linewidth}
\vspace{5mm}
  \centering
    \begin{tikzpicture}[scale=0.35]
	
	\draw[fill=black](0,0)circle(5pt);
        \draw[fill=black](3,0)circle(5pt);
        \draw[fill=black](6,0)circle(5pt);
        \draw[fill=black](9,0)circle(5pt);
        \draw[fill=black](0,3)circle(5pt);
        \draw[fill=black](3,3)circle(5pt);
        \draw[fill=black](6,3)circle(5pt);
        \draw[fill=black](9,3)circle(5pt);

        \draw[thick,blue](0,0)--(3,0)(6,0)--(9,0);
        \draw[thick,red](0,3)--(3,3)(6,3)--(9,3);	
    \end{tikzpicture}
\vspace{5mm}
\end{minipage}%

\\

\hline

\end{tabular}

\caption{The list of nice graphs with $4$ edges and no isolated vertices}
\label{table:nicegraphs}

\end{table}

\begin{definition}
    Let us call a graph $H$ \emph{nice} if it contains distinct edges $e_1,e_2,f_1,f_2$ such that $(e_1\cup e_2)\cap (f_1\cup f_2)\subset f_1\cap f_2$.
\end{definition}

\begin{remark}
    That the remaining $6$ graphs from Table \ref{table:4edgegraphs} are all nice is demonstrated in Table \ref{table:nicegraphs}. For each graph, suitable edges $e_1$ and $e_2$ are coloured red, while suitable $f_1$ and $f_2$ are coloured blue.
\end{remark}

\begin{theorem} \label{thm:nice graphs}
    If $H$ is a nice graph, then $g(n,H)=\Omega(n^{1/6})$.
\end{theorem}

\begin{proof}
    Choose distinct edges $e_1,e_2,f_1,f_2$ in $H$ such that $(e_1\cup e_2)\cap (f_1\cup f_2)\subset f_1\cap f_2$. For each $v\in V(K_n)$, let $f_v$ be a colouring of the edges of $K_n$ which uses $k\leq cn^{1/6}$ colours, where $c$ is a sufficiently small constant which depends on $H$. It suffices to prove that the collection of these colourings is not $(n,H)$-local.

    We need the following claim.

    \medskip

    \noindent \emph{Claim.}
    \begin{enumerate}[label=(\alph*)]
        \item There exist disjoint edges $p$ and $q$ in $K_n$ such that the number of vertices $v$ in $K_n$ with $f_v(p)=f_v(q)$ is $\Omega(n/k)$. \label{claim:disjoint}

        \item There exist distinct, intersecting edges $s$ and $t$ in $K_n$ such that the number of vertices $v$ in $K_n$ with $f_v(s)=f_v(t)$ is $\Omega(n/k)$. \label{claim:intersecting}
    \end{enumerate}

    \medskip

    \noindent \emph{Proof of Claim.} (a) It is easy to see by convexity that for any vertex $v\in V(K_n)$, there are $\Omega(n^4/k)$ pairs of disjoint edges $p$ and $q$ in $K_n$ such that $f_v(p)=f_v(q)$. Hence, the number of triples $(v,p,q)$ where $p$ and $q$ are disjoint edges in $K_n$ and $f_v(p)=f_v(q)$ is $\Omega(n^5/k)$. It follows from the pigeon hole principle that there exist $p$ and $q$ for which the number of suitable choices for $v$ is $\Omega(n/k)$.

    (b) It is easy to see by convexity that for any vertex $v\in V(K_n)$, there are $\Omega(n^3/k)$ pairs of distinct, intersecting edges $s$ and $t$ in $K_n$ such that $f_v(s)=f_v(t)$. Hence, the number of triples $(v,s,t)$ where $s$ and $t$ are distinct, intersecting edges in $K_n$ and $f_v(s)=f_v(t)$ is $\Omega(n^4/k)$. It follows from the pigeon hole principle that there exist $s$ and $t$ for which the number of suitable choices for $v$ is $\Omega(n/k)$. $\Box$

    \medskip
    
    Now if $e_1$ and $e_2$ are disjoint in $H$, let us use part \ref{claim:disjoint} of the claim to find disjoint edges $p$ and $q$ in $K_n$ such that there is a set $A$ of $\Omega(n/k)$ vertices in $V(K_n)\setminus (p\cup q)$ such that each $v\in A$ satisfies $f_v(p)=f_v(q)$. Similarly, if $e_1$ and $e_2$ intersect each other in $H$, then let us use part \ref{claim:intersecting} of the claim to find distinct, intersecting edges $p$ and $q$ in $K_n$ such that there is a set $A$ of $\Omega(n/k)$ vertices in $V(K_n)\setminus (p\cup q)$ such that each $v\in A$ satisfies $f_v(p)=f_v(q)$. Our aim is to find a copy of $H$ in $K_n$ in which $e_1$ and $e_2$ are mapped to $p$ and $q$ (in an arbitrary way), and all the remaining vertices of $H$ are mapped to $A$. By the definition of $A$, for any such embedding $T$ of $H$ we have that $T$ is not rainbow with respect to $f_v$ whenever $v\in V(T)\setminus (p\cup q)$ (since we have $f_v(p)=f_v(q)$ for any such vertex). We will embed the vertices in $(f_1\cup f_2)\setminus (e_1\cup e_2)$ into $K_n$ in a way that for every $v\in p\cup q$ the images of the edges $f_1$ and $f_2$ will have the same colour with respect to the colouring $f_v$. If we can do this, we obtain an embedding $T$ of $H$ such that for each $v\in V(T)$ the graph $T$ is not rainbow with respect to $f_v$, showing that the collection of colourings is not $(n,H)$-local.

    We consider two cases. First, assume that $f_1$ and $f_2$ are disjoint in $H$. In particular, $f_1\cup f_2$ is disjoint from $e_1\cup e_2$. Label each edge between two vertices of $A$ by its colours with respect to the colourings $f_v$ for all $v\in p\cup q$. Depending on whether $p$ and $q$ intersect or not, this labels the edges by a triple or quadruple of colours. In particular, there are at most $k^4$ possible labels. Since there are $\Omega(n^2/k^2)$ edges between two vertices of $A$, there will be a label that appears on $\Omega(n^2/k^6)$ different edges. If $c$ is sufficiently small, then we obtain two non-intersecting edges within $A$ with the same label. Choosing these two edges as the image of $f_1$ and $f_2$, and mapping all remaining vertices of $H$ arbitrarily to $A$, we obtain the desired embedding of $H$.

    The second case is when $f_1$ and $f_2$ intersect each other in $H$. If the common vertex of $f_1$ and $f_2$ belongs to $e_1\cup e_2$, then we have already mapped it to some vertex $x$; else let us choose an arbitrary vertex $x\in A$ as its image in the embedding. Labelling each $y\in A\setminus \{x\}$ by the colours of $f_v(xy)$ for all $v\in p\cup q$, there are at most $k^4$ possible labels, so if $c$ is sufficiently small, then there exist $y\neq z$ in $A\setminus \{x\}$ such that $f_v(xy)=f_v(xz)$ holds for all $v\in p\cup q$. Mapping $f_1$ to $xy$ and $f_2$ to $xz$, and mapping all remaining vertices of $H$ to arbitrary vertices in $A$, we obtain a suitable embedding of $H$.
\end{proof}

We are now in a position to complete the proof of Theorem \ref{thm:characterize}. Our Theorem \ref{thm:nice graphs} shows that whenever $H$ contains one of the graphs in Table \ref{table:nicegraphs} as a subgraph, we have $g(n,H)=\Omega(n^{1/6})$. As mentioned above, Cheng and Xu \cite{CX22} proved that $g(n,C_4)=\Omega(n^{1/3})$ and $g(n,P_4)=\Omega(n^{1/5})$. Observe that any graph with $4$ edges which is triangle-free and not the disjoint union of $P_3$ and $P_1$ is equal to $C_4$, $P_4$ or one of the graphs in Table \ref{table:nicegraphs}. This proves Theorem \ref{thm:characterize} for graphs $H$ with $4$ edges. Finally, observe that any graph with at least $5$ edges contains $C_4$, $P_4$ or one of the graphs in Table \ref{table:nicegraphs} as a subgraph (clearly, it suffices to verify this for graphs obtained by adding an edge to one of the following three graphs: the triangle with a pendant edge, the triangle with an isolated edge and the union of a $P_3$ and a $P_1$). This completes the proof of Theorem \ref{thm:characterize}.

\section{Bounds for the Erdős--Gyárfás function}\label{sec:EGy}
\subsection{The proof of Theorem~\ref{theorem_EGygeneral}}
In this subsection we prove Theorem~\ref{theorem_EGygeneral} about subpolynomial values for the Erdős--Gyárfás function. We begin by briefly discussing our approach. We will use Theorem~\ref{theorem_EGy43}, i.e., an appropriate colouring for $r=3$, to construct our colouring for $r=4$ (and larger values of $r$). Let $c$ denote the colouring of the triples provided by Theorem~\ref{theorem_EGy43}. Let us first try to colour the edges of $K_n^{(4)}$ by simply ignoring one of the vertices: $e$ will receive colour $c(e-v(e))$, where $v(e)$ is some special vertex of $e$ (and $e-v(e)$ is the set $e\setminus\{v(e)\}$). Note that we will need to describe a rule to choose the special vertex $v(e)$ that we ignore.

Now let us see when this approach provides a colouring satisfying the conditions. Observe that a copy of $K_{5}^{(4)}$ receives at least $4$ colours if and only if there is at most one colour repetition, i.e., at most one pair of edges of the $K_{5}^{(4)}$ share the same colour. Assume instead that in our colouring we have two pairs of edges $e,e'$ and $f,f'$ in some $K_{5}^{(4)}$ sharing the same colour. Since any edge misses exactly one vertex of this $K_5^{(4)}$, there is a vertex $p$ contained in the intersection $e\cap e'\cap f\cap f'$. If we can make sure that $p$ is the special vertex that we ignore in each of $e,e',f,f'$, then we get a contradiction from $c(e-p)=c(e'-p)$ and $c(f-p)=c(f'-p)$ by the definition of $c$.

To choose the special vertex $v(e)$, one simple method is to take an ordering of all of the vertices, and pick $v(e)$ to be the largest element of $e$. Of course, the point $p\in e\cap e'\cap f\cap f'$ does not in general have to be the largest element of $e,e',f,f'$. So, instead of taking just one ordering, we will take many total orders, and our final colouring will be the product of the colourings corresponding to each of these total orders. We will have to pick our list of total orders in a special way; we will make use of the following result of Hajnal (see~\cite{spencer1971minimal}).

\begin{theorem}[\cite{spencer1971minimal}]\label{theorem_totalorders}
	Let $k$ be a fixed positive integer. For any positive integer $n$ and any set $V$ of size $n$, we can find $M=O(\log\log n)$ total orders $<_1, \dots, <_M$ on $V$ such that whenever $a_1,\dots,a_k$ are distinct elements of $V$, then there is some $j$ (with $1\leq j\leq M)$ such that $a_i<_j a_1$ for $i=2,\dots,k$. In other words, $a_1$ is maximal among $a_1,\dots,a_k$ in at least one of our total orders.
\end{theorem}

We are now ready to prove Theorem~\ref{theorem_EGygeneral}.

\begin{proof}[Proof of Theorem~\ref{theorem_EGygeneral}]
	We show the statement by induction on $r$. The $r=3$ case is exactly Theorem~\ref{theorem_EGy43}, so assume that $r\geq 4$ and the statement holds for smaller values of $r$. Let $V$ be our set of $n$ vertices. By the induction hypothesis, we can pick a colouring $c$ of the $(r-1)$-element subsets of $V$ such that at least $r-1$ colours appear among any $r$ vertices and $c$ uses $e^{(\log n)^{2/5+o(1)}}$ colours. Furthermore, by Theorem~\ref{theorem_totalorders} applied for $k=r+1$, we can pick $M=O(\log\log n)$ total orders $<_1,\dots,<_M$ on $V$ such that for any distinct vertices $a_1,\dots,a_{r+1}\in V$, $a_1$ is maximal among these $r+1$ vertices in one of the total orders $<_j$. Given a (non-empty) set $W\subseteq V$, let us write $\operatorname{max}_j W$ for the element of $W$ which is maximal in $W$ in the ordering $<_j$. We define a colouring $c'$ of $r$-element subsets of $V$ by setting
	$$c'(e)=(c(e-\operatorname{max}_1(e)),c(e-\operatorname{max}_2(e)),\dots,c(e-\operatorname{max}_M(e))).$$
	
	In other words, $c_j$ is the product of all the colourings $c'_j(e)=c(e-\operatorname{max}_j(e))$ formed by using the colouring $c$ after ignoring the largest element of $e$ in the ordering $<_j$. Note that the number of colours used is at most
	$$\left(e^{(\log n)^{2/5+o(1)}}\right)^{O(\log \log n)}=e^{(\log n)^{2/5+o(1)}}.$$
	
	We claim that in any $K_{r+1}^{(r)}$, at least $r$ colours appear under the colouring $c'$. Assume, for contradiction, that this is not the case. Then there exist a set $W$ of $r+1$ vertices and subsets $e,e',f,f'\subseteq W$ of size $r$ (with $e\not =e', f\not =f'$) such that $c'(e)=c'(e')$, $c'(f)=c'(f')$ and $\{e,e'\}\not =\{f,f'\}$. Since any $r$-edge inside $W$ misses exactly one element of $W$, we have $|e\cap e'\cap f\cap f'|\geq (r+1)-4\geq 1$. Pick any element $p\in e\cap e'\cap f\cap f'$, and let $j\in\{1,\dots, M\}$ be such that $p=\operatorname{max}_j(W)$.
	
	Since $c'(e)=c'(e')$, we have (by taking $j$th coordinates) $c(e-p)=c(e'-p)$. Similarly, $c(f-p)=c(f'-p)$. But this means that, under the colouring $c$, at most $r-2$ different colours appear in the $K_{r}^{(r-1)}$ induced by $W-p$. This contradicts our choice of $c$ and finishes the proof.
\end{proof}

\subsection{Connections to Problem~\ref{prob:Wig}}\label{subsection_connectiontoWig}
Recall that lower bounds on $g(n,P_3)$ also imply the same lower bounds for Problem~\ref{prob:Wig}. However, we have seen (Theorem~\ref{thm:P3}) that $g(n,P_3)$ is subpolynomial, so this method cannot give a polynomial lower bound for the problem of Karchmer and Wigderson.

It is very natural to try to give lower bounds to Problem~\ref{prob:Wig} by considering only specific forms of triples $x,y,z$; in particular, it is natural to take $(x\cup y\cup z)\setminus (x\cap y\cap z)$ to be small to make sure that only a few conditions need to be satisfied. Thus, instead of taking $x,y,z$ to be edges of a $P_3$, we could add to each of them the same set of vertices, i.e., take $x=x_0\cup S, y=y_0\cup S, z=z_0\cup S$, where $x_0,y_0,z_0$ are sets of size $2$ forming a $P_3$.

However, Theorem~\ref{theorem_EGygeneral} implies that this cannot work for sets $S$ of some given bounded size $\ell$. Indeed, we can proceed similarly as for $P_3$s. Let $\gamma$ be a colouring coming from Theorem~\ref{theorem_EGygeneral} for $r=\ell+3$, and define $f_v(x)$ to be $\gamma(\{v\}\cup x)$ if $v\not \in x$ and some arbitrary different colour if $v\in x$. Then, similarly to the proof of Theorem~\ref{thm:P3}, we find that whenever $x=x_0\cup S$, $y=y_0\cup S$ and $z=z_0\cup S$ such that $|S|=\ell$ and $x_0,y_0,z_0$ form a $P_3$ $abcd$ (disjoint from $S$), then $x,y,z$ receive distinct colours in either $f_a$ or $f_d$.

More generally, similar arguments can be used to construct colourings $f_v$ for each $v\in V(K_n)$ using a subpolynomial number of colours such that whenever $x,y,z$ are distinct subsets of $V(K_n)$ of bounded size, then there exists some $v\in (x\cup y\cup z)\setminus (x\cap y\cap z)$ for which $f_v(x)$, $f_v(y)$ and $f_v(z)$ are distinct.

Note, however, that the argument above does not work when $|S|$ is large (say, at least logarithmic in $n$); perhaps considering such $S$ might give better bounds.

\section{Concluding remarks} \label{sec:concluding}
In this paper we have determined, for each graph $H$ other than the disjoint union of $P_3$ and $P_1$ with an arbitrary number of isolated vertices, whether the growth of $g(n,H)$ is polynomial. The natural question left open by our investigations is whether $g(n,H)$ is polynomial when $H$ is $P_3\cup P_1$.

Problem \ref{prob:Wig} also remains open, with the best known lower bound being $g(n,P_3)=\Omega((\frac{\log n}{\log \log n})^{1/4})$ due to Alon and Ben-Eliezer \cite{AB11}. Our Theorem \ref{thm:P3} shows that this approach cannot give a polynomial lower bound. Indeed, as we mentioned in Subsection \ref{subsection_connectiontoWig}, one cannot obtain a polynomial lower bound by considering only vectors of bounded Hamming weight. On the other hand, one can prove a polynomial upper bound using the asymmetric Lov\'asz Local Lemma. It remains an interesting open question whether the answer to Problem \ref{prob:Wig} is polynomial.

\bibliographystyle{abbrv}
\bibliography{bibliography}

\end{document}